\documentclass[final,5p,times,twocolumn]{elsarticle}

\usepackage{amssymb}
\usepackage{amsthm}
\usepackage{graphicx}
\usepackage[cmex10]{amsmath}
\usepackage{array}
\usepackage{mdwmath}
\usepackage{mdwtab}
\usepackage{fixltx2e}
\usepackage{graphics} 
\usepackage{epsfig} 
\usepackage{times} 

\usepackage{url}
\usepackage{lscape}
\usepackage{color}
\usepackage{epsfig}
\usepackage{amsthm} 
\newtheorem{theorem}{Theorem}

\theoremstyle{definition}

\newtheorem{remark}{Remark}
\usepackage{hyperref}
\newtheorem{assumption}{Assumption}

\newtheorem{problem}{Problem} 

\begin{document}

\begin{frontmatter}


\title{Practical Considerations for Implementing \\ Robust-to-Early Termination Model Predictive Control\tnoteref{label1}}
\tnotetext[label1]{This research has been supported by National Science Foundation under award number DGE-2244082, and by WSU Voiland College of Engineering and Architecture through a start-up package to M. Hosseinzadeh.}

\author{Mohsen Amiri}
\ead{mohsen.amiri@wsu.edu}

\author{Mehdi Hosseinzadeh\corref{cor}}
\ead{mehdi.hosseinzadeh@wsu.edu}

\address{School of Mechanical and Materials Engineering, Washington State University, Pullman, WA 99164, USA}

\cortext[cor]{Corresponding author.}


\begin{abstract}
Model Predictive Control (MPC) is widely used to achieve performance objectives, while enforcing operational and safety constraints. Despite its high performance, MPC often demands significant computational resources, making it challenging to implement in systems with limited computing capacity. A recent approach to address this challenge is to use the Robust-to-Early Termination (REAP) strategy. At any time instant, REAP converts the MPC problem into the evolution of a virtual dynamical system whose trajectory converges to the optimal solution, and provides guaranteed sub-optimal and feasible solution whenever its evolution is terminated due to limited computational power. REAP has been introduced as a continuous-time scheme and its theoretical properties have been derived under the assumption that it performs all the computations in continuous time. However, REAP should be practically implemented in discrete-time. This paper focuses on the discrete-time implementation of REAP, exploring conditions under which anytime feasibility and convergence properties are maintained when the computations are performed in discrete time. The proposed methodology is validated and evaluated through extensive simulation and experimental studies.
\end{abstract}


\begin{keyword}
Model Predictive Control\sep Robust-to-Early Termination\sep Discrete-Time Implementation \sep Limited Computing Capacity.


\end{keyword}

\end{frontmatter}

\section{Introduction}\label{sec:Introduction}

Designing control laws that can be implemented in real-world systems, particularly those subject to constraints, has become a significant challenge in recent decades. Constrained control is primarily addressed by Model Predictive Control (MPC) \cite{rawlings2017model, camacho2007model}, which optimizes performance objectives by solving an optimization problem at every sampling instant within a receding horizon, while ensuring all operational and safety constraints are satisfied. However, the computational demands arising from online optimization and diverse state and input constraints pose significant challenges, in particular, in systems with fast dynamics and/or limited computational resources.

Different approaches have been introduced to overcome the computational challenge of MPC implementation. One approach is to pre-compute optimal laws offline and store them for online use, as exploited in explicit MPC  \cite{alessio2009survey,kvasnica2019complexity,xiu2020grid}; however, this approach requires high memory for large problems and is not robust to early termination of the searching process. Another approach involves triggering mechanisms, used in self-triggered \cite{henriksson2012self,long2022safety,lu2020self} and event-triggered MPC \cite{mi2018event,wang2020event,heydari2022robust}; yet, limited computational power can prevent online optimization from converging when triggered. Another approach is to perform a fixed number of optimization iterations \cite{cimini2017exact,ghaemi2009integrated,liao2020time}; in general, this approach does not ensure constraint satisfaction at all times. Anytime MPC \cite{feller2017stabilizing,feller2018sparsity,gharbi2021anytime} converts the MPC problem into an unconstrained optimization problem using logarithmic barrier functions, ensuring stability with minimal iterations; however, in general, it does not guarantee constraint satisfaction. Another approach introduced in \cite{SSMPC,Limon2008,Ferramosca2009} modifies the MPC formulation such that one can arbitrarily decrease the length of the prediction horizon (and consequently reduce the size of the problem) without violating closed-loop stability and recursive feasibility; note that this approach cannot ensure constraint satisfaction in the event of early termination of the corresponding optimization problem. Another approach is dynamically embedded MPC \cite{nicotra2018dynamically}, which embeds the solution of the MPC problem into the internal state vector of a dynamical feedback law; note that, in general, dynamically embedded MPC does not guarantee constraint satisfaction at all times. Dynamically embedded MPC has been enhanced in \cite{nicotra2018embedding} by augmenting it with the Explicit Reference Governor (ERG) \cite{hosseinzadeh2019explicit,hosseinzadeh2019constrained} to guarantee constraint satisfaction; however, the solution is usually conservative due to the natural conservatism of ERG.

The most recent and promising approach to address the limitations of computational resources for MPC implementation is to utilize the Robust-to-Early terminAtion oPtimization (REAP) which has been introduced in \cite{Hosseinzadeh2023RobustTermination}. REAP exploits the primal-dual gradient flow \cite{BoydBook,HosseinzadehTAC2020,Hosseinzadeh2022_ROTEC} to embed the optimal solution of the MPC problem into the internal states of a virtual continuous-time dynamical system that runs in parallel to the process. At any time instant, REAP runs until available time for execution runs out and the solution is guaranteed to be sub-optimal and feasible despite early termination. The main advantage of REAP over the above-mentioned efforts is that it provides a feasible solution even when the time available for online computation of the MPC solution at any time instant is limited and varies over time.




One of the main obstacles to a widespread use of REAP is that it has been introduced as a continuous-time scheme and its theoretical properties have been derived under the assumption that all computations are performed in continuous time. However, REAP should be practically implemented in discrete time, and its not clear if its properties are maintained when computations are performed in discrete time. This paper aims at addressing the discrete-time implementation of REAP. The core idea is to adaptively adjust the evolution rate of the gradient flow dynamics such that REAP's properties are maintained when REAP computes the updates of the decision variables in discrete time. Theoretical properties of such a discrete-time implementation are proven analytically, and extensive simulation and experimental studies are provided to support our analyses.  Moreover, a MATLAB package has been developed to allow researchers to utilize REAP in their research to address limitations on computational resources for MPC implementations.


The rest of this paper is organized as follows. Section \ref{sec:ProblemStatement} provides preliminaries and states the problem. Section \ref{sec:Discrete-Time} develops the discrete-time REAP and proves its theoretical properties. Section \ref{sec:Results} assesses the performance of the proposed methodology through extensive simulation and experimental studies on a Parrot Bebop 2 drone. Finally, Section \ref{sec:Conclusion} concludes the paper.

\medskip
\noindent\textbf{Notation:} We use $\mathbb{R}$ to denote the set of real numbers, $\mathbb{R}_{>0}$ to denote the set of positive real numbers, and $\mathbb{R}_{\geq0}$ to denote the set of non-negative real numbers. We use $\mathbb{Z}_{>0}$ denotes the set of positive integer numbers. We use $x\in\mathbb{R}^n$ to indicate that $x$ belongs to $n$-dimensional real space. Given $x\in\mathbb{R}^n$, we denote its transpose by $x^\top$. Also, $\left\Vert x\right\Vert_Q^2=x^\top Qx$, where $Q\in\mathbb{R}^{n\times n}$. We use $\ast$ in the superscript to indicate optimal decision. We use $\text{diag}\{a_1,a_2,\cdots,a_n\}$ to denote a $n\times n$ matrix with diagonal entries $a_1,a_2,\cdots,a_n\in\mathbb{R}$. We use $I_n$ to denote $n\times n$ identity matrix. 

\section{Preliminaries and Problem statement}\label{sec:ProblemStatement}

Consider the following discrete-time linear time-invariant system:
\begin{subequations}\label{eq:System1}
\begin{align}
x(t+1)=&A x(t)+B u(t),\\
y(t)=&C x(t)+D u(t),
\end{align}    
\end{subequations}
where $ x(t) = [x_1(t) \; \ldots \; x_n(t)]^\top \in \mathbb{R}^n $ is the state vector at time instant $ t $, $ u(t) = [u_1(t) \; \ldots \; u_p(t)]^\top \in \mathbb{R}^p $ is the control input at time instant $ t $, $ y(t) = [y_1(t) \; \ldots \; y_m(t)]^\top \in \mathbb{R}^m $ is the output vector at time instant $ t $, and $ A \in \mathbb{R}^{n \times n}, B \in \mathbb{R}^{n \times p}, C \in \mathbb{R}^{m \times n} $, and $ D \in \mathbb{R}^{m \times p} $ are system matrices.
\begin{assumption}\label{Assum:Controllability}
The pair $(A, B)$ is stabilizable and the pair $(C, A)$ is detectable.
\end{assumption}

Suppose that states and inputs of system \eqref{eq:System1} are constrained at all times as follows:
\begin{subequations}\label{eq:constraints}
\begin{align}
x(t)&\in\mathcal{X}\subset\mathbb{R}^n,\label{eq:constraints1}\\
u(t)&\in\mathcal{U}\subset\mathbb{R}^p,\label{eq:constraints2}
\end{align}   
\end{subequations}
where $ \mathcal{X} $ and $ \mathcal{U} $ are convex polytopes defined as:
\begin{subequations} \label{eq:allConstraints}
\begin{align}
& \mathcal{X} = \left\{ x \in \mathbb{R}^n : a_i^\top x + b_i \leq 0, \; i = 1, \ldots, c_x \right\},\label{eq:Xconstraints} \\
& \mathcal{U} = \left\{ u \in \mathbb{R}^p : c_i^\top u + d_i \leq 0, \; i = 1, \ldots, c_u \right\},\label{eq:Uconstraints} 
\end{align}
\end{subequations}
with $ a_i \in \mathbb{R}^n$, $b_i \in \mathbb{R}$, $c_i \in \mathbb{R}^p$, $d_i \in \mathbb{R} $, $ c_x $ being the number of constraints on the state vector, and $ c_u $ being the number of constraints on the control input.


Let $ r \in \mathbb{R}^m $ be the desired reference, and suppose that the rank of the matrix $\left[\begin{matrix}A-I_n & B\\C & D\end{matrix}\right]$ equals the rank of the augmented matrix $\left[\begin{matrix}A-I_n & B & \mathbf{0}\\C & D & r\end{matrix}\right]$. Thus, there exist \cite{StrangBook} at least one steady-state configuration $(\bar{x_r},\bar{u_r})$ satisfying
\begin{subequations}\label{eq:SSconfiguration}
 \begin{align}
\bar{x}_r &= A \bar{x}_r + B \bar{u}_r,\\
r &= C \bar{x}_r + D \bar{u}_r.
\end{align}   
\end{subequations}

If $ \bar{x}_r \in (1-\epsilon)\mathcal{X} $ and $ \bar{u}_r \in (1-\epsilon)\mathcal{U}$ for some $\epsilon\in(0,1)$,  the reference signal $r$ is called a strictly steady-state admissible reference. We denote the set of all strictly steady-state admissible references by $\mathcal{R}\subset\mathbb{R}^m$.

For the given $r\in\mathcal{R}$, we define the Region of Attraction (RoA) as the set of all initial conditions $x(0)$ within $\mathcal{X}$ that allow the steady-state configuration $(\bar{x}_r,\bar{u}_r)$ to be achieved without violating the constraints given in \eqref{eq:constraints}. We use $\text{RoA}_r$ to denote the RoA corresponding to the steady-state admissible reference $r\in\mathcal{R}$. It is obvious that $\text{RoA}_r\subset\mathcal{X}$.

\subsection{MPC Formulation}
Given $ r \in \mathcal{R} $ as the desired reference, $x(0)\in\text{RoA}_r$ as the initial condition, and $ N \in \mathbb{Z}_{>0} $ as the length of the prediction horizon, MPC computes the optimal control sequence over the prediction horizon $\mathbf{u}^\ast(t) := \left[ \left(u^\ast(0|t)\right)^{\top}, \ldots, \left(u^\ast(N-1|t)\right)^{\top} \right]^{\top} \in \mathbb{R}^{Np} $ by solving the following optimization problem:
\begin{subequations}\label{eq:OptimizationProblemExtend}
\begin{align}
\arg\,\min _{\mathbf{u}} & \sum_{k=0}^{N-1}\left\|\hat{x}(k|t)-\bar{x}_r\right\|_{Q_x}^2+\sum_{k=0}^{N-1}\left\|u(k|t)-\bar{u}_r\right\|_{Q_u}^2 \nonumber\\
& +\left\|\hat{x}(N|t)-\bar{x}_r\right\|_{Q_N}^2,\label{eq:CostFunctionExtend}
\end{align}
subject to
\begin{align}
& \hat{x}(k+1|t)=A \hat{x}(k|t)+B u(k),~\hat{x}(0|t)=x(t), \label{eq:ConstraintRevised}\\
& \hat{x}(k|t) \in \mathcal{X}, k \in\{0, \cdots, {\color{red}~~~~~~~~~N}\},\\
& u(k|t) \in \mathcal{U}, k \in\{0,\cdots, N-1\}, \\
& (\hat{x}(N|t), r) \in \Omega,\label{eq:ConstraintTerminalExtend}
\end{align}
\end{subequations}
where $k$ indicates the time instant along the prediction horizon, $Q_x=Q_x^{\top} \succeq 0\left(Q_x \in \mathbb{R}^{n \times n}\right), Q_u=Q_u^{\top} \succ 0\left(Q_u \in \mathbb{R}^{p \times p}\right)$, $Q_N\succ0$ ($Q_N \in$ $\mathbb{R}^{n \times n}$) is the terminal cost weight (see Subsection \ref{sec:TerminalCostMatrix}), and $\Omega$ is the terminal constraint set (see Subsection \ref{sec:TerminalConstraintSet}).

\subsubsection{Terminal Cost Weight $Q_N$}\label{sec:TerminalCostMatrix}
In MPC formulation, it is convenient \cite{Hosseinzadeh2023RobustTermination,nicotra2018dynamically,nicotra2018embedding} to choose $Q_N$ in \eqref{eq:CostFunctionExtend} as the solution of the following algebraic Riccati equation: $Q_N = A^{\top} Q_N A - \left(A^{\top} Q_N B\right) \left(Q_u + B^{\top} Q_N B\right)^{-1} \left(B^{\top} Q_N A\right) + Q_x$.

\subsubsection{Terminal Constraint Set $\Omega$}\label{sec:TerminalConstraintSet}
Consider the terminal control law $\kappa(x(t), r) = \bar{u}_r + K(x(t) - \bar{x}_r)$, where $K$ is such that $A + BK$ is Schur. We select the feedback gain matrix $K$ in the terminal control law as $K = -\left(Q_u + B^{\top} Q_N B\right)^{-1} \left(B^{\top} Q_N A\right)$. Note that such a selection is optimal for the unconstrained problem \cite{nicotra2018dynamically,Hosseinzadeh2023RobustTermination}.




As discussed in \cite{SSMPC}, terminal constraint given in \eqref{eq:ConstraintTerminalExtend} can be implemented by the following constraints:
\begin{subequations}\label{eq:TerminalConstraintsAll}
\begin{align}
& (A + BK)^\omega \hat{x}(N|t) + \sum_{j=1}^\omega (A + BK)^{j-1} \left(B\bar{u}_r - BK\bar{x}_r\right) \in \mathcal{X}, \\
& K \hat{x}(N|k) - K \bar{x}_r + \bar{u}_r \in \mathcal{U},\\
& K(A + BK)^\omega \hat{x}(N|t) + K\sum_{j=1}^\omega (A + BK)^{j-1} \left(B\bar{u}_r - BK\bar{x}_r\right) \nonumber\\
&~~~~~~~~~~~~~~~~~~~~~~~~~~~~~~~~+\left(B\bar{u}_r - BK\bar{x}_r\right) \in \mathcal{U},
\end{align}
\end{subequations}
where $\omega \in \{1, \ldots, \omega^\ast\}$ with $\omega^\ast$ being a finite index that can be computed by solving a series of offline mathematical programming problems as detailed in \cite{Gilbert1991}.

\subsection{Final MPC Problem}
It is obvious that constraint set $\mathcal{X}$ given in \eqref{eq:Xconstraints} and terminal constraints given in \eqref{eq:TerminalConstraintsAll} can be framed as constraints on the control sequence $\mathbf{u}$. Thus, the MPC problem can be expressed as:
\begin{subequations}\label{eq:OptimizationProblemMain}
\begin{align}
\arg\,\min _{\mathbf{u}} & \sum_{k=0}^{N-1}\left\|\hat{x}(k|t)-\bar{x}_r\right\|_{Q_x}^2+\sum_{k=0}^{N-1}\left\|u(k|t)-\bar{u}_r\right\|_{Q_u}^2 \nonumber\\
& +\left\|\hat{x}(N|t)-\bar{x}_r\right\|_{Q_N}^2,\label{eq:CostFunction}
\end{align}
subject to constraint \eqref{eq:ConstraintRevised} and
\begin{align}
\mathbf{u}\in\mathcal{\mathbf{U}} = \left\{\mathbf{u} :  \eta_i^{\top} \mathbf{u} + \gamma_i  \leq 0, \, i = 1, \ldots, \bar{c} \right\},\label{eq:ConstraintTerminal}
\end{align}
\end{subequations}
where $\mathcal{\mathbf{U}}\subset\mathcal{U}\subset\mathbb{R}^p$, $\mathbf{u}:= \left[ \left(u(0|t)\right)^{\top}, \ldots, \left(u(N-1|t)\right)^{\top} \right]^{\top} \in \mathbb{R}^{Np} $, and $\bar{c}$ indicates the total number of constraints (i.e., $\bar{c}=c_x+c_u+c_\Omega$), with $ c_\Omega $ being the number of constraints in \eqref{eq:TerminalConstraintsAll}.

\subsection{Robust-to-Early Termination MPC}
Let $\mathcal{\mathbf{U}}_\beta$ be a tightened version of the set $\mathbf{U}$, defined as:
\begin{align}\label{eq:tightenedConstraints}
    \mathcal{\mathbf{U}}_\beta = \left\{\mathbf{u}:\eta_i^{\top}\mathbf{u}+ \gamma_i + \frac{1}{\beta} \leq 0, \, i = 1, \ldots, \bar{c} \right\},
\end{align}
where $\beta > 0$. Note that, for any given reference $r$, tightening the set $\mathcal{\mathbf{U}}$ by $\beta$ to construct the set $\mathcal{\mathbf{U}}_\beta$ as in \eqref{eq:tightenedConstraints} can reduce the region of attraction. However, for any reference $r\in\mathcal{R}$ and the initial condition $x(0)\in\text{RoA}_r$ for which the optimization problem \eqref{eq:OptimizationProblemMain} is feasible and $\mathbf{u}(t)$ remains within the interior of the set $\mathcal{\mathbf{U}}$ for all $t$, one can select a sufficiently large $\beta$ to ensure that the optimization problem with the constraint $\mathbf{u}\in\mathcal{\mathbf{U}}_\beta$ is also feasible.


Let $\mathbf{u}^{\dag}(t)$ represent the solution of the optimization problem \eqref{eq:OptimizationProblemMain} with the tightened constraint set $\mathcal{\mathbf{U}}_\beta$. Notably, $\mathbf{u}^{\dag}(t)\rightarrow \mathbf{u}^\ast(t)$ as $\beta\rightarrow\infty$.

The modified barrier function \cite{Polyak1992,Melman1996,Vassiliadis1998} associated with the optimization problem \eqref{eq:OptimizationProblemMain} with the tightened constraint set given in \eqref{eq:tightenedConstraints} is:
\begin{align}\label{eq:BarrierFunction}
&\mathcal{B}\big(x(t),r,\mathbf{u},\lambda\big)=\sum_{k=0}^{N-1} \left\|\hat{x}(k|t) - \bar{x}_r \right\|_{Q_x}^2 + \sum_{k=0}^{N-1} \left\|u(k|t) - \bar{u}_r \right\|_{Q_u}^2 \nonumber\\
&+ \left\|\hat{x}(N|t) - \bar{x}_r \right\|_{Q_N}^2 - \sum_{i=1}^{\bar{c}} \lambda_{i} \log \Big(-\beta \big(\eta_i^{\top} \mathbf{u} + \gamma_i  + 1/\beta\big) + 1 \Big),
\end{align}
where $\lambda = \left[\lambda_{1}, \cdots, \lambda_{\bar{c}}\right]^{\top} \in \mathbb{R}_{\geq 0}^{\bar{c}}$ 
is the vector of dual variables. At any time instant $t$, we denote the vector of the optimal dual variables by $\lambda^{\dag}(t)$.

At any time instant $t$, REAP uses the primal-dual gradient flow \cite{Feijer2010,HosseinzadehTAC,HosseinzadehLetter,LucaLetter,HosseinzadehTAC2024} to inform the following virtual continuous-time dynamical system:
\begin{subequations}\label{eq:DualGradientFlow}
\begin{align}
\frac{d}{d s} \hat{\mathbf{u}}_c(s|t) &= -\sigma \nabla_{\hat{\mathbf{u}}_c} \mathcal{B}\big(x(t), r, \hat{\mathbf{u}}_c(s|t), \hat{\lambda}_c(s|t)\big), \\
\frac{d}{d s} \hat{\lambda}_c(s|t) &= +\sigma \Big(\nabla_{\hat{\lambda}_c} \mathcal{B}\big(x(t), r, \hat{\mathbf{u}}_c(s|t), \hat{\lambda}_c(s|t)\big) + \Phi(s|t) \Big),
\end{align}
\end{subequations}
where the subscript ``c" indicates that the signal is continuous-time, $\sigma \in \mathbb{R}_{>0}$ is called the Karush-Kuhn-Tucker (KKT) parameter that determines the evolution rate of system \eqref{eq:DualGradientFlow}, $s$ is the auxiliary time variable that REAP spends solving the MPC problem, and $\Phi(s) \in \mathbb{R}^{\bar{c}}$  is the projection operator whose $i$th entry, denoted by $[\Phi(s|t)]_i$, is:
\begin{align}\label{eq:PhiDefinition}
[\Phi(s|t)]_i = \left\{\begin{array}{rl}
0, & \text{if } \hat{\lambda}_{c_i}(s|t) > 0 \text{ or } \big(\hat{\lambda}_{c_i}(s|t) = 0 \\
& \text{ and } \left[\nabla_{\hat{\lambda}_c} \mathcal{B}(\cdot) \right]_i \geq 0 \big), \\
-\left[\nabla_{\hat{\lambda}_c} \mathcal{B}(\cdot) \right]_i, & \text{otherwise},
    \end{array}\right.,
\end{align}
with $\hat{\lambda}_{c_i}(s|t)$ and $\left[\nabla_{\hat{\lambda}_c} \mathcal{B}(\cdot) \right]_i$ representing the $i$th entries of $\hat{\lambda}_c(s|t)$ and $\mathcal{B}\big(x(t), r, \hat{\mathbf{u}}_c(s|t), \hat{\lambda}_c(s|t)\big)$, respectively. The differential equations given in \eqref{eq:DualGradientFlow} constitute a virtual continuous-time dynamical system, which should be deployed at any time $t$.

Starting from a feasible initial condition $\left(\hat{\mathbf{u}}_c(0|t),\hat{\lambda}_c(0|t)\right)$,  REAP provides the following properties at any time instant $t$: i) as $s\rightarrow\infty$, $\left(\hat{\mathbf{u}}_c(s|t),\hat{\lambda}_c(s|t)\right)$ converges to $\left(\mathbf{u}^\dag(t),\lambda^\dag(t)\right)$ exponentially fast and at a tunable convergence rate; and ii) $\hat{\mathbf{u}}_c(s|t)\in\mathbf{U}_\beta$ for all $s$, where $\mathbf{U}_\beta$ is as in \eqref{eq:tightenedConstraints}. The significance of the above-mentioned second property (i.e., anytime feasibility) is that the evolution of system \eqref{eq:DualGradientFlow} can be stopped at any moment with a guaranteed feasible solution. Thus, REAP runs until available time for execution runs out (that may not be known in advance) and the solution is guaranteed to be sub-optimal and enforce the constraints despite any early termination.

\begin{remark}\label{Remark:InitialFeasibility}
A feasible $\hat{\textbf{u}}_c(0|t)$ control sequence is always available, and can be constructed based on the control sequence computed at time instant $t-1$, which is shifted by one and padded with the terminal control law. As discussed in \cite{Hosseinzadeh2023RobustTermination}, such a selection is effective, as states of the system do not change substantially from one time instant to the next in most applications. At time instant $t=0$, since MPC given in \eqref{eq:OptimizationProblemMain} adapts to changes in the desired reference, the feasibility of $\hat{\mathbf{u}}_c(0|0)$ can be ensured by using the  feasibility governor detailed in \cite{nicotra2018embedding}. Regarding $\hat{\lambda}_c(0|t)$, note that any non-negative value for $\hat{\lambda}_{c_i}(0|t),~i\in\{1,\cdots,\bar{c}\}$ is feasible.
\end{remark}

\subsection{Goal}
REAP has the remarkable feature of providing a sub-optimal and feasible solution whenever its evolution is terminated. These properties are guaranteed when REAP performs the computations in continuous time. However, REAP should be practically implemented in discrete-time. Note that although using an exact discretization method should preserve these properties when the auxiliary system \eqref{eq:DualGradientFlow} is implemented in discrete time, since the auxiliary system \eqref{eq:DualGradientFlow} is nonlinear (as $\nabla_{\hat{\lambda}_c}\mathcal{B}(\cdot)$ and $\nabla_{\hat{\mathbf{u}}_c}\mathcal{B}(\cdot)$ are nonlinear), exact discretization can be quite challenging and may increase computational demands.

Prior work \cite{Hosseinzadeh2023RobustTermination} suggests implementing REAP by using the Euler method with a sufficiently small discretization step, without discussing how to analytically obtain an appropriate discretization step. Note that using an excessively small discretization step decreases the evolution rate of the virtual system \eqref{eq:DualGradientFlow} such that it may not converge in the available time. Since REAP uses an acceptance/rejection mechanism that applies the backup solution consisting of ``shifted by 1 and added by the terminal control law'' control sequence from the previous step if system \eqref{eq:DualGradientFlow} does not converge, using a very small discretization step can degrade feedback and robustness to disturbances and measurement noise; it can also degrade the performance of the closed-loop system by incorporating the non-optimal terminal control law into the considered control sequence. Therefore, it is essential to examine the theoretical guarantees for the discrete-time implementation of REAP and ensure its convergence and anytime feasibility properties. This paper aims at achieving this objective through addressing the following problem.


\begin{problem}\label{Problem}
Given $r\in\mathcal{R}$, suppose that REAP give in \eqref{eq:DualGradientFlow} is implemented in discrete time by using the Euler method, as follows
\begin{subequations}\label{eq:EulerMethod}
\begin{align}
\hat{\mathbf{u}}(\tau|t) =&\hat{\mathbf{u}}(\tau-1|t)-\sigma\cdot d\tau\cdot \nabla_{\hat{\mathbf{u}}} \mathcal{B}\big(x(t), r, \hat{\mathbf{u}}(\tau-1|t), \hat{\lambda}(\tau-1|t)\big), 
\\
\hat{\lambda}(\tau|t) = & \hat{\lambda}(\tau-1|t)+\sigma\cdot d\tau\cdot  \Big(\nabla_{\hat{\lambda}} \mathcal{B}\big(x(t), r, \hat{\mathbf{u}}(\tau-1|t), \hat{\lambda}(\tau-1|t)\big) \nonumber\\
&+ \Phi(\tau-1|t) \Big),
\end{align}
\end{subequations}
where $\tau$ indicates the computation step,  and $d\tau$ is the discretization step. Provide conditions under which the discrete-time implementation as in \eqref{eq:EulerMethod} maintains REAP's convergence and anytime feasibility properties. 
 \end{problem}


\section{Discrete-Time Implementation of REAP}\label{sec:Discrete-Time}
Suppose that between MPC time instants (i.e., during the time interval $[t,t+1)$) during which the state $x(t)$ and the desired reference $r$ are constant, we use REAP to obtain the optimal solution $\mathbf{u}^\ast(t)$. When REAP is implemented in continuous time, its anytime feasibility property has been developed based upon the fact that as the trajectories of the system \eqref{eq:DualGradientFlow} approach the boundary of the following sets that are constant during the time interval $[t,t+1)$:
\begin{subequations}\label{eq:ssconstraint}
\begin{align}
\mathcal{D}_{\hat{\mathbf{u}}}=&\left\{\hat{\mathbf{u}} \big| \eta_i^{\top} \hat{\mathbf{u}} + \gamma_i \leq 0, i=1, \cdots, \bar{c}\right\},\\
\mathcal{D}_{\hat{\lambda}}=&\left\{\hat{\lambda} \big| \hat{\lambda}_i\geq 0, i=1, \cdots, \bar{c}\right\},
\end{align}
\end{subequations}
there will be a sufficient repulsion pushing the trajectories towards the interior of these sets. Thus, the sets $\mathcal{D}_{\hat{\mathbf{u}}}$ and $\mathcal{D}_{\hat{\lambda}}$ remain invariant at any time instant $t$ during REAP's execution. However, when updates of the primal and dual variables are computed in discrete time as in \eqref{eq:EulerMethod}, discretization errors can compromise the invariance of the set $\mathcal{D}_{\hat{\mathbf{u}}}$ and/or $\mathcal{D}_{\hat{\lambda}}$. This issue is illustrated geometrically in Figure \ref{fig:constraint violation} for the primal variable $\hat{\mathbf{u}}$. One potential mitigation strategy involves using a small $\beta$. However, this approach is not desired, as a small $\beta$ will impact optimality by increasing the difference between $\mathbf{u}^\ast(t)$ and $\mathbf{u}^\dag(t)$, and will degrade the performance by classifying safe control sequences as unsafe ones.

According to \eqref{eq:EulerMethod}, another strategy to avoid violations due to discretization errors is to limit the changes in the primal and dual variables, though this can reduce the evolution rate of system \eqref{eq:DualGradientFlow}. 
Limiting $\|\hat{\mathbf{u}}(\tau|t) - \hat{\mathbf{u}}(\tau-1|t)\|$ and $\|\hat{\lambda}(\tau|t) - \hat{\lambda}(\tau-1|t)\|$ can be achieved through either a saturation function (similar to the one reported in \cite{hosseinzadeh2019explicit,hosseinzadeh2019constrained}) or by selecting $d\tau$ and $\sigma$ such that $d\tau \cdot \sigma$ is sufficiently small. The use of a saturation function is not practical, as: i) determining a safe saturation level with minimal performance degradation is not straightforward; ii) a safe saturation level for one scenario may be unsafe for another; and iii) using a saturation function can unnecessarily degrade performance when constraints are far from being violated.

The second approach (making $d\tau \cdot \sigma$ sufficiently small) is also challenging, as: i) there is no analytical way to determine values of $\sigma$ and $d\tau$ necessary to avoid violations due to discretization errors (note that values that are safe for one scenario might be unsafe for another); and ii) when constraints are far from being violated, and thus
large changes in the primal and dual parameters are allowed, reducing $d\tau \cdot \sigma$ can unnecessarily prevent these changes, and consequently, degrade the performance.

\begin{figure}[t]
    \centering
    \includegraphics[width=.65\linewidth]{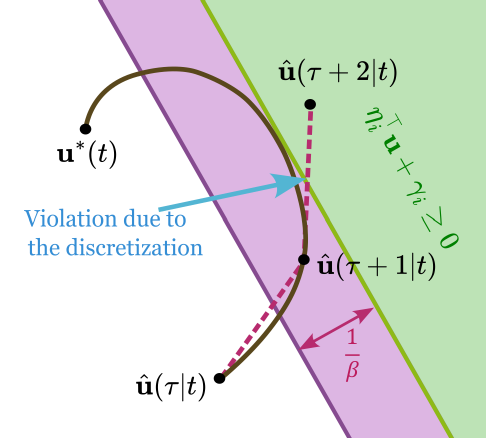}
    \caption{Comparison of the Continuous and discrete implementation of REAP regarding constraint violation.}
    \label{fig:constraint violation}
\end{figure}

Instead of using a constant $\sigma$ and $d\tau$ (and consequently a constant $d\tau\cdot \sigma$), this paper proposes to fix the value of $d\tau$ (i.e., to use a constant discretization step, which is relevant) and use a dynamic KKT parameter (i.e., $\sigma$) where the value of $\sigma$ at any computation step is selected so as to prevent constraint violation due to discretization errors. More precisely, we propose the following discrete-time implementation:
\begin{subequations}\label{eq:DisREAP}
\begin{align}
\hat{\mathbf{u}}(\tau|t) =&\hat{\mathbf{u}}(\tau-1|t)-\sigma(\tau|t)\cdot d\tau\cdot \nabla_{\hat{\mathbf{u}}} \mathcal{B}\big(x(t), r, \hat{\mathbf{u}}(\tau-1|t),\nonumber\\
&\hat{\lambda}(\tau-1|t)\big), \label{eq:DisREAP1}
\\
\hat{\lambda}(\tau|t) = & \hat{\lambda}(\tau-1|t)+\sigma(\tau|t)\cdot d\tau\cdot  \Big(\nabla_{\hat{\lambda}} \mathcal{B}\big(x(t), r, \hat{\mathbf{u}}(\tau-1|t),\nonumber\\
&\hat{\lambda}(\tau-1|t)\big)+ \Phi(\tau-1|t) \Big).\label{eq:DisREAP2}
\end{align}
\end{subequations}

\begin{figure*}\setcounter{equation}{15}
 \begin{subequations}\label{eq:Sigmas}
\begin{align}
\sigma_i^{\hat{\mathbf{u}}}(\tau|t)=&\frac{\delta_i^{\hat{\mathbf{u}}}(\tau-1|t)}{d\tau\cdot\left\Vert\eta_i\right\Vert\cdot\max \left\{\left\Vert\nabla_{\hat{\mathbf{u}}} \mathcal{B}\big(x(t), r, \hat{\mathbf{u}}(\tau-1|t), \hat{\lambda}(\tau-1|t)\big)\right\Vert, \psi\right\}},\\
\sigma_i^{\hat{\lambda}}(\tau|t)=&\frac{\delta_i^{\hat{\lambda}}(\tau-1|t)}{d\tau\cdot\max \left\{\left\vert\left[\nabla_{\hat{\lambda}} \mathcal{B}\big(x(t), r, \hat{\mathbf{u}}(\tau-1|t), \hat{\lambda}(\tau-1|t)\big)\right]_i+\left[\Phi(\tau-1|t)\right]_i\right\vert, \psi\right\}}.
\end{align}
\end{subequations}   
\hrule
\end{figure*}

The following theorem provides a condition on the KKT parameter $\sigma(\tau|t)$ under which the invariance of the sets $\mathcal{D}_{\hat{\mathbf{u}}}$ and $\mathcal{D}_{\hat{\lambda}}$ given in \eqref{eq:ssconstraint} is guaranteed during the REAP's execution, when the iterates are computed as in \eqref{eq:DisREAP}.


\begin{theorem}\label{Theorem1}
Let $\hat{\mathbf{u}}(\tau-1|t)\in\mathcal{D}_{\hat{\mathbf{u}}}$ and $\hat{\lambda}(\tau-1|t)\in\mathcal{D}_{\hat{\lambda}}$. Let the KKT parameter $\sigma(\tau|t)$ be selected such that the following condition is satisfied:\setcounter{equation}{14}
\begin{align}\label{eq:SigmaCondition}
0\leq\sigma(\tau|t)\leq\min\left\{\min_{i\in\{1,\cdots,\bar{c}\}}\left\{\sigma_i^{\hat{\mathbf{u}}}(\tau|t)\right\},\min_{i\in\{1,\cdots,\bar{c}\}}\left\{\sigma_i^{\hat{\lambda}}(\tau|t)\right\}\right\},
\end{align}
where $\sigma_i^{\hat{\mathbf{u}}}(\tau|t)$ and $\sigma_i^{\hat{\lambda}}(\tau|t)$ are as in \eqref{eq:Sigmas}, $\big[\nabla_{\hat{\lambda}} \mathcal{B}\big(x(t), r,\hat{\mathbf{u}}(\tau-1|t),\hat{\lambda}(\tau-1|t)\big)\big]_i$ and $\big[\Phi(\tau-1|t)\big]_i$ are respectively the $i$th entry of $\nabla_{\hat{\lambda}} \mathcal{B}\big(x(t), r,\hat{\mathbf{u}}(\tau-1|t),\hat{\lambda}(\tau-1|t)\big)$ and $\Phi(\tau-1|t)$, $\psi\in\mathbb{R}_{>0}$ is a smoothing factor to prevent numerical issues, $\delta_i^{\hat{\mathbf{u}}}$ is the Euclidean distance between $\hat{\mathbf{u}}(\tau-1|t)$ and the $i$-th constraint computed as:\setcounter{equation}{16}
\begin{align}\label{eq:distance}
\delta_i^{\hat{\mathbf{u}}}(\tau-1|t)=& \left\{ 
\begin{array}{ll}
& \min\limits_{\theta} \left\|\hat{\mathbf{u}}(\tau-1|t) - \theta\right\| \\
\text{s.t.} & \eta_i^{\top} \theta + \gamma_i+\epsilon=0
\end{array}
\right.,
\end{align}
and $\delta_i^{\hat{\lambda}}(\tau-1|t)=\hat{\lambda}_i(\tau-1|t)-\epsilon$, with $\epsilon$ being a small positive constant (see Remark \ref{remark:Epsilon}). Then, $\hat{\mathbf{u}}(\tau|t) \in \mathcal{D}_{\hat{\mathbf{u}}}$ and $\hat{\lambda}(\tau|t) \in \mathcal{D}_{\hat{\lambda}}$.
\end{theorem}

\begin{proof}
The Cauchy–Schwarz inequality implies that:
\begin{align}\label{eq:Comment1}
\left\Vert\eta_i^{\top}\big(\hat{\mathbf{u}}(\tau|t) - \hat{\mathbf{u}}(\tau-1|t)\big)\right\Vert \leq\left\Vert\eta_i\right\Vert \left\Vert\hat{\mathbf{u}}(\tau|t) - \hat{\mathbf{u}}(\tau-1|t)\right\Vert,
\end{align}
for all $i\in\{1,\cdots,\bar{c}\}$. Thus, according to \eqref{eq:DisREAP1}, it follows from \eqref{eq:Comment1} that:
\begin{align}\label{eq:Comment2}
&\left\Vert\eta_i^{\top}\big(\hat{\mathbf{u}}(\tau|t) - \hat{\mathbf{u}}(\tau-1|t)\big)\right\Vert \leq\left\Vert\eta_i\right\Vert \cdot d\tau \cdot \sigma(\tau|t) \cdot \bigg\Vert \nabla_{\hat{\mathbf{u}}} \mathcal{B}\big(x(t), r,\nonumber\\
&~~~~~~~~~~~~~~~~~~~~\hat{\mathbf{u}}(\tau-1|t),\hat{\lambda}(\tau-1|t)\big)\bigg\Vert\nonumber\\
&\leq\left\Vert\eta_i\right\Vert \cdot d\tau \cdot \sigma(\tau|t) \cdot \max\Big\{\Big\Vert \nabla_{\hat{\mathbf{u}}} \mathcal{B}\big(x(t), r,\nonumber\\
&~~~~~~~~~~~~~~~~~~~~\hat{\mathbf{u}}(\tau-1|t),\hat{\lambda}(\tau-1|t)\big)\Big\Vert,\psi\Big\},
\end{align}
for any $\psi\in\mathbb{R}_{>0}$. Also, from \eqref{eq:DisREAP2}, we have:
\begin{align}\label{eq:Comment3}
&\left\vert\hat{\lambda}_i(\tau|t) - \hat{\lambda}_i(\tau-1|t)\big)\right\vert \leq d\tau \cdot \sigma(\tau|t)\cdot\Big\vert\big[\nabla_{\hat{\lambda}} \mathcal{B}\big(x(t), r,\hat{\mathbf{u}}(\tau-1|t),\nonumber\\
&~~~~~~~~~~~~~~~~~~~~\hat{\lambda}(\tau-1|t)\big)\big]_i+\big[\Phi(\tau-1|t)\big]_i\Big\vert\nonumber\\
&\leq d\tau \cdot \sigma(\tau|t)\cdot\max\Big\{\Big\vert\big[\nabla_{\hat{\lambda}} \mathcal{B}\big(x(t), r,\hat{\mathbf{u}},(\tau-1|t),\hat{\lambda}(\tau-1|t)\big)\big]_i\nonumber\\
&~~~~~~~~~~~~~~~~~~~~+\big[\Phi(\tau-1|t)\big]_i\Big\vert,\psi\Big\},
\end{align}
for all $i\in\{1,\cdots,\bar{c}\}$.

Suppose that $\hat{\mathbf{u}}(\tau-1|t)\in\mathcal{D}_{\hat{\mathbf{u}}}$ and $\hat{\lambda}(\tau-1|t)\in\mathcal{D}_{\hat{\lambda}}$, and let $\delta_i^{\hat{\mathbf{u}}}(\tau-1|t)$ and $\delta_i^{\hat{\lambda}}(\tau-1|t)$ be defined as above. Suppose that $\sigma(\tau|t)$ is selected such that $\left\Vert\eta_i\right\Vert \cdot d\tau \cdot \sigma(\tau|t) \cdot \max\Big\{\Big\Vert \nabla_{\hat{\mathbf{u}}} \mathcal{B}\big(x(t), r,\hat{\mathbf{u}}(\tau-1|t),\hat{\lambda}(\tau-1|t)\big)\Big\Vert,\psi\Big\}\leq \delta_i^{\hat{\mathbf{u}}}(\tau-1|t)$ and $d\tau \cdot \sigma(\tau|t)\cdot\max\Big\{\Big\vert\big[\nabla_{\hat{\lambda}} \mathcal{B}\big(x(t), r,\hat{\mathbf{u}},(\tau-1|t),\hat{\lambda}(\tau-1|t)\big)\big]_i+\big[\Phi(\tau-1|t)\big]_i\Big\vert,\psi\Big\}
\leq\delta_i^{\hat{\lambda}}(\tau-1|t)$ for all $i\in\{1,\cdots,\bar{c}\}$. Thus, inequalities \eqref{eq:Comment2} and \eqref{eq:Comment3} imply that $\left\Vert\eta_i^{\top}\big(\hat{\mathbf{u}}(\tau|t) - \hat{\mathbf{u}}(\tau-1|t)\big)\right\Vert\leq\delta_i^{\hat{\mathbf{u}}}(\tau-1|t)$ and $\left\vert\hat{\lambda}_i(\tau|t) - \hat{\lambda}_i(\tau-1|t)\big)\right\vert\leq\delta_i^{\hat{\lambda}}(\tau-1|t)$ for all $i\in\{1,\cdots,\bar{c}\}$; in other words, $\hat{\mathbf{u}}(\tau|t)\in\mathcal{D}_{\hat{\mathbf{u}}}$ and $\hat{\lambda}(\tau|t)\in\mathcal{D}_{\hat{\lambda}}$. See Figure \ref{fig:delta} and \eqref{fig:lambda} for geometric illustrations.

Therefore, if $\sigma(\tau|t)$ satisfies condition \eqref{eq:SigmaCondition} at every step $\tau$, the sets $\mathcal{D}_{\hat{\mathbf{u}}}$ and $\mathcal{D}_{\hat{\lambda}}$ remain invariant when the primal and dual variables are updated in discrete time as in \eqref{eq:DisREAP}.
\end{proof}

\begin{remark}\label{remark:Epsilon}
The constant $\epsilon$ define in \eqref{eq:distance} ensures that $\eta_i^{\top} \hat{\mathbf{u}}(\tau|t)+ \gamma_i <\epsilon$, and consequently, $\log \Big(-\beta \big(\eta_i^{\top} \hat{\mathbf{u}}(\tau|t)+ \gamma_i  + 1/\beta\big) + 1 \Big)>0$ for all $\tau$. Therefore, $\mathcal{B}(\cdot)$, and $\nabla_{\hat{\mathbf{u}}} \mathcal{B}(\cdot)$ and $\nabla_{\hat{\lambda}} \mathcal{B}(\cdot)$ are both bounded and definite. This means that the dynamical model given in \eqref{eq:DisREAP} is implementable.
\end{remark}

\begin{remark}
Considering the worst-case scenario (i.e., using the Euclidean distance from the constraints and taking the minimum over all as in \eqref{eq:SigmaCondition}) can lead to a conservative solution, even though it ensures the invariance of the sets $\mathcal{D}_{\hat{\mathbf{u}}}$ and $\mathcal{D}_{\hat{\lambda}}$. Taking into account the direction of trajectory of system \eqref{eq:DisREAP} can help reduce this conservatism. Future work will explore this approach to reduce the conservatism of the scheme.
\end{remark}

\begin{remark}\label{remark:nonzero}
At time instant $t$, when $\delta_i^{\hat{\mathbf{u}}}(\tau-1|t)=0$ and/or $\delta_i^{\hat{\lambda}}(\tau-1|t)=0$, it can be concluded from \eqref{eq:Sigmas} that $\sigma(\tau|t)$ becomes zero. Hence, according to \eqref{eq:DisREAP}, $\hat{\mathbf{u}}(\tau+1|t) =\hat{\mathbf{u}}(\tau|t)$ and $\hat{\lambda}(\tau+1|t)=\hat{\lambda}(\tau|t)$, implying that $\sigma(\tau+\vartheta|t)=0,~\forall \vartheta=0,1,\cdots$. As a result, the evolution of the primal and dual variables will stall for time instant $t$. At time instant $t+1$, by constructing the initial condition $\hat{\mathbf{u}}(0|t+1)$ as discussed in Remark \ref{Remark:InitialFeasibility}, $\delta_i^{\hat{\mathbf{u}}}(\tau-1|t)$ and  $\delta_i^{\hat{\lambda}}(\tau-1|t)$ would be nonzero, and thus the primal and dual variables will adjust toward the optimal values. This points will be demonstrated through simulation and experimental results in Section \ref{sec:Results}. 
\end{remark}

\begin{figure}[t]
    \centering
    \includegraphics[width=.65\linewidth]{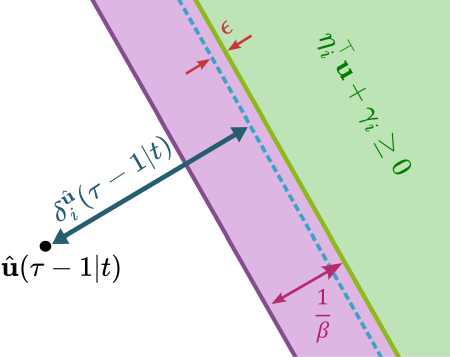}
    \caption{Geometric illustration of how to determine the KKT parameter $\sigma(\tau|t)$ to maintain the invariance of the set $\mathcal{D}_{\hat{\mathbf{u}}}$.}
    \label{fig:delta}
\end{figure}

\begin{figure}[t]
    \centering
    \includegraphics[width=.8\linewidth]{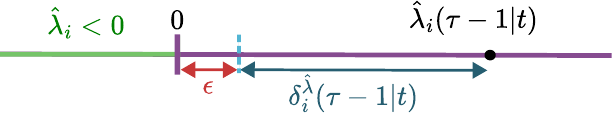}
    \caption{Geometric illustration of how to determine the KKT parameter $\sigma(\tau|t)$ to maintain the invariance of the set $\mathcal{D}_{\hat{\lambda}}$.}
    \label{fig:lambda}
\end{figure}

\subsection{Theoretical Properties}
This subsection analytically proves theoretical properties of the discrete-time REAP given in \eqref{eq:DisREAP}. Similar to the continuous-time scheme detailed in \cite{Hosseinzadeh2023RobustTermination}, we will employ Lyapunov theory and Alexandrov's theorem \cite[p. 333]{Alexandrov2005} to prove the convergence and anytime feasibility of the discrete-time REAP \eqref{eq:DisREAP}, respectively. First, we show that updating the primal and dual variables in discrete time as in \eqref{eq:DisREAP} does not violate boundedness and convergence of REAP.

\begin{theorem}
Let $\big(\hat{\mathbf{u}}(\tau|t), \hat{\lambda}(\tau|t)\big)$ be the trajectory of \eqref{eq:DisREAP}. Then, starting from the initial condition $\big(\hat{\mathbf{u}}(0|t), \hat{\lambda}(0|t)\big)$ such that $\hat{\mathbf{u}}(0|t)\in\mathcal{D}_{\hat{\mathbf{u}}}$ and $\hat{\lambda}(0|t)\in\mathcal{D}_{\hat{\lambda}}$, the following properties hold:
\begin{itemize}
    \item The following set is invariant:
    \begin{align}\label{eq:SetT}
        \mathcal{T}=\Big\{\big(\hat{\mathbf{u}},\hat{\lambda}\big)\big|&\left\Vert\hat{\mathbf{u}}-\mathbf{u}^{\dag}(t)\right\Vert^2+\left\Vert\hat{\lambda}-\lambda^{\dag}(t)\right\Vert^2\leq\left\Vert\hat{\mathbf{u}}(0|t)-\mathbf{u}^{\dag}(t)\right\Vert^2\nonumber\\
        &+\left\Vert\hat{\lambda}(0|t)-\lambda^{\dag}(t)\right\Vert^2\Big\};
    \end{align}
    that is, $\hat{\mathbf{u}}(\tau|t)$ and $\hat{\lambda}(\tau|t)$ are bounded for all $\tau$.  

    \item if $\sigma(\tau|t)$ is not equivalently equal to zero at all $\tau$, $\big(\hat{\mathbf{u}}(\tau|t), \hat{\lambda}(\tau|t)\big)$  converges to $\left(\mathbf{u}^{\dag}(t), \lambda^{\dag}(t)\right)$ as $\tau\rightarrow\infty$.
\end{itemize}

\end{theorem}

\begin{proof}
Since the modified barrier function $\mathcal{B}\big(x(t), r, \hat{\mathbf{u}}(\tau|t), \hat{\lambda}(\tau|t)\big)$ is strongly convex with respect to the primal variable $\hat{\mathbf{u}}$, following arguments similar to \cite{Hosseinzadeh2023RobustTermination,Hosseinzadeh2022_ROTEC,Ryu2016} implies that there exists $\mu\in\mathbb{R}_{>0}$ such that:
\begin{align}\label{eq:prrofofNabla}
\left[\begin{array}{c}
\nabla_{\hat{\mathbf{u}}} \mathcal{B}(\cdot) \\
-\nabla_{\hat{\lambda}} \mathcal{B}(\cdot)
\end{array}\right]^\top\left[\begin{array}{c}
\hat{\mathbf{u}}(\tau|t)-\mathbf{u}^{\dag}(t) \\
\hat{\lambda}(\tau|t)-\lambda^{\dag}(t)
\end{array}\right]\geq \mu\cdot\left\|\left[\begin{array}{c}
\hat{\mathbf{u}}(\tau|t)-\mathbf{u}^{\dag}(t) \\
\hat{\lambda}(\tau|t)-\lambda^{\dag}(t)
\end{array}\right]\right\|^2,
\end{align}
where $\nabla_{\hat{\mathbf{u}}}\mathcal{B}(\cdot)$ and $\nabla_{\hat{\lambda}}\mathcal{B}(\cdot)$ indicate gradient of the modified barrier function with respect to the primal variable $\hat{\mathbf{u}}$ and dual variable $\hat{\lambda}$, respectively. Note that inequality \eqref{eq:prrofofNabla} holds true whether the gradients are evaluated at step $\tau$ or step $\tau-1$.


At this stage, we consider the following Lyapunov function:
\begin{align}
W\big(\hat{\mathbf{u}}(\tau|t),\hat{\lambda}(\tau|t)\big)=\left[\begin{matrix}
\hat{\mathbf{u}}(\tau|t)-\mathbf{u}^{\dag}(t) \\
\hat{\lambda}(\tau|t)-\lambda^{\dag}(t)
\end{matrix}\right]^{\top}\left[\begin{matrix}
\hat{\mathbf{u}}(\tau|t)-\mathbf{u}^{\dag}(t) \\
\hat{\lambda}(\tau|t)-\lambda^{\dag}(t)
\end{matrix}\right],
\end{align}
whose time difference according to \eqref{eq:DisREAP} is\footnote{For the sake of brevity, we denote $W\big(\hat{\mathbf{u}}(\tau|t),\hat{\lambda}(\tau|t)\big)$ and $W\big(\hat{\mathbf{u}}(\tau-1|t),\hat{\lambda}(\tau-1|t)\big)$ by $W(\tau|t)$ and $W(\tau-1|t)$, respectively.}:
\begin{align}\label{eq:TimeDifference1}
&\Delta W(\tau|t):=W(\tau|t)-W(\tau-1|t)\nonumber\\
&=\left[\begin{matrix}
\hat{\mathbf{u}}(\tau|t)-\mathbf{u}^{\dag}(t) \\
\hat{\lambda}(\tau|t)-\lambda^{\dag}(t)
\end{matrix}\right]^{\top}\left[\begin{matrix}
\hat{\mathbf{u}}(\tau|t)-\mathbf{u}^{\dag}(t)\\
\hat{\lambda}(\tau|t)-\lambda^{\dag}(t)
\end{matrix}\right]\nonumber\\
&-\left[\begin{matrix}
\hat{\mathbf{u}}(\tau-1|t)-\mathbf{u}^{\dag}(t) \\
\hat{\lambda}(\tau-1|t)-\lambda^{\dag}(t)
\end{matrix}\right]^{\top}\left[\begin{matrix}
\hat{\mathbf{u}}(\tau-1|t)-\mathbf{u}^{\dag}(t) \\
\hat{\lambda}(\tau-1|t)-\lambda^{\dag}(t)
\end{matrix}\right]\nonumber\\
&=\left[\begin{matrix}
\hat{\mathbf{u}}(\tau|t)-\mathbf{u}^{\dag}(t) \\
\hat{\lambda}(\tau|t)-\lambda^{\dag}(t)
\end{matrix}\right]^{\top}\left[\begin{matrix}
\hat{\mathbf{u}}(\tau-1|t)-\mathbf{u}^{\dag}(t)\\
\hat{\lambda}(\tau-1|t)-\lambda^{\dag}(t)
\end{matrix}\right]\nonumber\\
&+\left[\begin{matrix}
\hat{\mathbf{u}}(\tau|t)-\mathbf{u}^{\dag}(t) \\
\hat{\lambda}(\tau|t)-\lambda^{\dag}(t)
\end{matrix}\right]^{\top}\left[\begin{matrix}
-\sigma(\tau|t)\cdot d\tau\cdot \nabla_{\hat{\mathbf{u}}}\mathcal{B}(\tau-1|t)\\
\sigma(\tau|t)\cdot d\tau\cdot  \big(\nabla_{\hat{\lambda}}\mathcal{B}(\tau-1|t)+ \Phi(\tau-1|t) \big)
\end{matrix}\right]\nonumber\\
&-\left[\begin{matrix}
\hat{\mathbf{u}}(\tau-1|t)-\mathbf{u}^{\dag}(t) \\
\hat{\lambda}(\tau-1|t)-\lambda^{\dag}(t)
\end{matrix}\right]^{\top}\left[\begin{matrix}
\hat{\mathbf{u}}(\tau-1|t)-\mathbf{u}^{\dag}(t) \\
\hat{\lambda}(\tau-1|t)-\lambda^{\dag}(t)
\end{matrix}\right],
\end{align}
where $\nabla_{\hat{\mathbf{u}}} \mathcal{B}(\tau-1|t)$ and $\nabla_{\hat{\lambda}}\mathcal{B}(\tau-1|t)$ indicate $\nabla_{\hat{\mathbf{u}}}\mathcal{B}\big(x(t), r, \hat{\mathbf{u}}(\tau-1|t), \hat{\lambda}(\tau-1|t)\big)$ and $\nabla_{\hat{\lambda}}\mathcal{B}\big(x(t), r, \hat{\mathbf{u}}(\tau-1|t), \hat{\lambda}(\tau-1|t)\big)$, respectively. From \eqref{eq:TimeDifference1}, we have:
\begin{align}\label{eq:TimeDifference2}
&\Delta W(\tau|t)=\left[\begin{matrix}
\hat{\mathbf{u}}(\tau-1|t)-\mathbf{u}^{\dag}(t)\\
\hat{\lambda}(\tau-1|t)-\lambda^{\dag}(t)
\end{matrix}\right]^\top\left[\begin{matrix}
\hat{\mathbf{u}}(\tau|t)-\hat{\mathbf{u}}(\tau-1|t)\\
\hat{\lambda}(\tau|t)-\hat{\lambda}(\tau-1|t)
\end{matrix}\right]\nonumber\\
&+\left[\begin{matrix}
\hat{\mathbf{u}}(\tau|t)-\mathbf{u}^{\dag}(t) \\
\hat{\lambda}(\tau|t)-\lambda^{\dag}(t)
\end{matrix}\right]^{\top}\left[\begin{matrix}
-\sigma(\tau|t)\cdot d\tau\cdot \nabla_{\hat{\mathbf{u}}}\mathcal{B}(\tau-1|t)\\
\sigma(\tau|t)\cdot d\tau\cdot  \big(\nabla_{\hat{\lambda}} \mathcal{B}(\tau-1|t)+ \Phi(\tau-1|t) \big)
\end{matrix}\right],
\end{align}
which according to \eqref{eq:DisREAP} implies that:
\begin{align}\label{eq:TimeDifference3}
&\Delta W(\tau|t)=-\sigma(\tau|t)\cdot d\tau\cdot\left[\begin{matrix}
 \nabla_{\hat{\mathbf{u}}}\mathcal{B}(\tau-1|t)\\
-\nabla_{\hat{\lambda}} \mathcal{B}(\tau-1|t)-\Phi(\tau-1|t)
\end{matrix}\right]^\top\cdot\nonumber\\
&\left(\left[\begin{matrix}
\hat{\mathbf{u}}(\tau|t)-\mathbf{u}^{\dag}(t) \\
\hat{\lambda}(\tau|t)-\lambda^{\dag}(t)
\end{matrix}\right]+\left[\begin{matrix}
\hat{\mathbf{u}}(\tau-1|t)-\mathbf{u}^{\dag}(t)\\
\hat{\lambda}(\tau-1|t)-\lambda^{\dag}(t)
\end{matrix}\right]\right).
\end{align}

From \eqref{eq:PhiDefinition}, it can be concluded that $\big(\hat{\lambda}(\tau|t)-\lambda^{\dag}(t)\big)^\top\big(\nabla_{\hat{\lambda}} \mathcal{B}(\tau|t)+ \Phi(\tau|t) \big)\leq\big(\hat{\lambda}(\tau|t)-\lambda^{\dag}(t)\big)^\top\nabla_{\hat{\lambda}} \mathcal{B}(\tau|t)$. Thus, it follows from \eqref{eq:TimeDifference3} that:
\begin{align}\label{eq:TimeDifference4}
&\Delta W(\tau|t)\leq-\sigma(\tau|t)\cdot d\tau\cdot\left[\begin{matrix}
 \nabla_{\hat{\mathbf{u}}}\mathcal{B}(\tau-1|t)\\
-\nabla_{\hat{\lambda}} \mathcal{B}(\tau-1|t)
\end{matrix}\right]^\top\cdot\nonumber\\
&\left(\left[\begin{matrix}
\hat{\mathbf{u}}(\tau|t)-\mathbf{u}^{\dag}(t) \\
\hat{\lambda}(\tau|t)-\lambda^{\dag}(t)
\end{matrix}\right]+\left[\begin{matrix}
\hat{\mathbf{u}}(\tau-1|t)-\mathbf{u}^{\dag}(t)\\
\hat{\lambda}(\tau-1|t)-\lambda^{\dag}(t)
\end{matrix}\right]\right),
\end{align}

Using the inequality \eqref{eq:prrofofNabla} twice in \eqref{eq:TimeDifference4} yields:
\begin{align}\label{eq:TimeDifference5} 
\Delta W(\tau|t)\leq-\sigma(\tau|t)\cdot d\tau\cdot\mu\cdot\Bigg(\left\|\left[\begin{array}{c}
\hat{\mathbf{u}}(\tau|t)-\mathbf{u}^{\dag}(t) \\
\hat{\lambda}(\tau|t)-\lambda^{\dag}(t)
\end{array}\right]\right\|^2\nonumber\\
+\left\|\left[\begin{array}{c}
\hat{\mathbf{u}}(\tau-1|t)-\mathbf{u}^{\dag}(t) \\
\hat{\lambda}(\tau-1|t)-\lambda^{\dag}(t)
\end{array}\right]\right\|^2\Bigg),
\end{align}
which implies that the set $\mathcal{T}$ given in \eqref{eq:SetT} is invariant; that is, $\big(\hat{\mathbf{u}}(\tau|t), \hat{\lambda}(\tau|t)\big)\in\mathcal{T}$ for all $\tau$, and thus, $\hat{\mathbf{u}}(\tau|t)$ and $\hat{\lambda}(\tau|t)$ are bounded for all $\tau$.

For what regards the second property, when $\sigma(\tau|t)$ does not equivalently remain at zero, it follows from \eqref{eq:TimeDifference5} that $\Delta W(\tau|t)<0$ for all $\big(\hat{\mathbf{u}}(\tau|t), \hat{\lambda}(\tau|t)\big)\neq\left(\mathbf{u}^{\dag}(t), \lambda^{\dag}(t)\right)$. This implies the convergence of $\big(\hat{\mathbf{u}}(\tau|t), \hat{\lambda}(\tau|t)\big)$ to $\left(\mathbf{u}^{\dag}(t), \lambda^{\dag}(t)\right)$ as $\tau\rightarrow\infty$. 
\end{proof}

Next, we show that updating the primal and dual variables in discrete time as in \eqref{eq:DisREAP} does not violate the anytime feasibility property of REAP.

\begin{theorem}\label{Theorem3}
Consider the discrete-time REAP given in \eqref{eq:DisREAP}.Let $\hat{\mathbf{u}}(0|t)\in\mathcal{D}_{\hat{\mathbf{u}}}$ and $\hat{\lambda}(0|t)\in\mathcal{D}_{\hat{\lambda}}$. Then, $\hat{\mathbf{u}}(\tau|t)$ for all $\tau$ satisfies the constraints given in \eqref{eq:ConstraintTerminal}, if the KKT parameter $\sigma(\tau|t)$ satisfies the condition given in \eqref{eq:SigmaCondition}. 
\end{theorem}

\begin{proof}  
As discussed in \cite{Hosseinzadeh2023RobustTermination}, $\mathcal{B}\big(x(t), r, \hat{\mathbf{u}}(\tau|t), \hat{\lambda}(\tau|t)\big)\rightarrow\infty$ only if $\eta_i^{\top} \hat{\mathbf{u}} + \gamma_i\rightarrow0^-$ for one (or more) $i$. Thus, the boundedness of $\mathcal{B}\big(x(t), r, \hat{\mathbf{u}}(\tau|t), \hat{\lambda}(\tau|t)\big)$ from above is equivalent to the constraint satisfaction for all $\tau$.

First, incorporating $\epsilon$ in computing the Euclidean distances from the primal and dual trajectories and the constraints as in \eqref{eq:distance} ensures that the modified barrier function $\mathcal{B}\big(x(t),r,\hat{\mathbf{u}}(\tau|t), \hat{\lambda}(\tau|t)\big)$ remains definite for all $\tau$ (see Remark \ref{remark:Epsilon}). Second, Theorem \ref{Theorem1} implies that the set $\mathcal{D}_{\hat{\mathbf{u}}}$ is invariant, and thus, $\hat{\mathbf{u}}(\tau|t)$ cannot jump the boundary of the set $\mathcal{D}_{\hat{\mathbf{u}}}$. Third, following arguments similar to \cite{Hosseinzadeh2023RobustTermination} and utilizing the Alexandrov's theorem \cite[p. 333]{Alexandrov2005}, it can be shown that as $\hat{\mathbf{u}}(\tau|t)$ approach the boundary of the set $\mathcal{D}_{\hat{\mathbf{u}}}$, $\mathcal{B}\big(x(t), r, \hat{\mathbf{u}}(\tau+1|t), \hat{\lambda}(\tau+1|t)\big)<\mathcal{B}\big(x(t), r, \hat{\mathbf{u}}(\tau|t), \hat{\lambda}(\tau|t)\big)$. 

Therefore, it can be concluded that the modified barrier function $\mathcal{B}\big(x(t),r,\hat{\mathbf{u}}(\tau|t), \hat{\lambda}(\tau|t)\big)$ is always definite and decreases along the trajectories of $\hat{\mathbf{u}}(\tau|t)$ when these trajectories are near the boundary of the  set $\mathcal{D}_{\hat{\mathbf{u}}}$, which completes the proof.


\end{proof}

\begin{remark}
The proposed discrete-time REAP may yield a conservative solution. The primary sources of this conservatism are: i) the tightening factor $\beta$ introduced in \eqref{eq:tightenedConstraints}; ii) the consideration of worst-case scenarios regarding the distance from the current primal and dual variables to the constraints; and iii) the asymptotic convergence of the auxiliary system in \eqref{eq:DisREAP}. Understanding the sources of conservatism allows us to employ the following methods, which are not mutually exclusive, to mitigate it: i) using a larger $\beta$; and ii) implementing an appropriate warm-starting scheme (such as the ones reported in \cite{Hosseinzadeh2023RobustTermination} and \cite{Wang2010}) to enhance convergence. 
\end{remark}

\subsection{DiscreteREAP Toolbox}
To broaden the impact of our work and make the proposed methodology publicly available to researcher, we have developed a MATLAB package, that is available at the following URL: \href{https://github.com/mhsnar/DiscreteREAP.git}{https://github.com/mhsnar/DiscreteREAP.git}. Details of the developed package are provided in Appendix. 

\begin{figure}[t]
\begin{center}
\includegraphics[width=\linewidth]{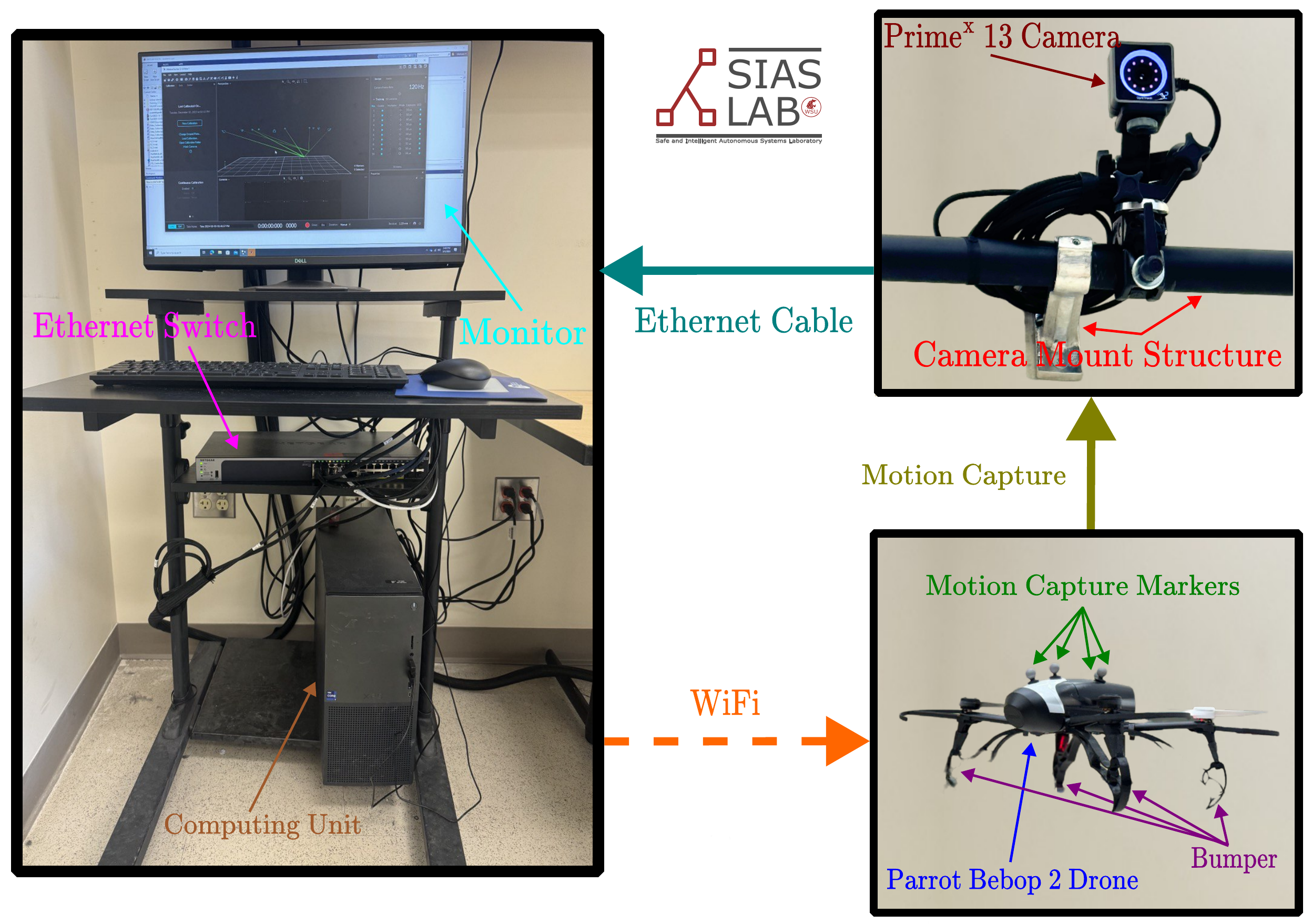}  \\
\caption{Overview of the experimental setup utilized to perform the experimental evaluation.}
\label{fig:network}
\end{center}
\end{figure}

\begin{figure*}[t]
    \centering
     \includegraphics[width=.75\linewidth]{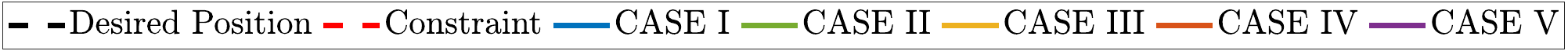}\\ \includegraphics[width=\linewidth]{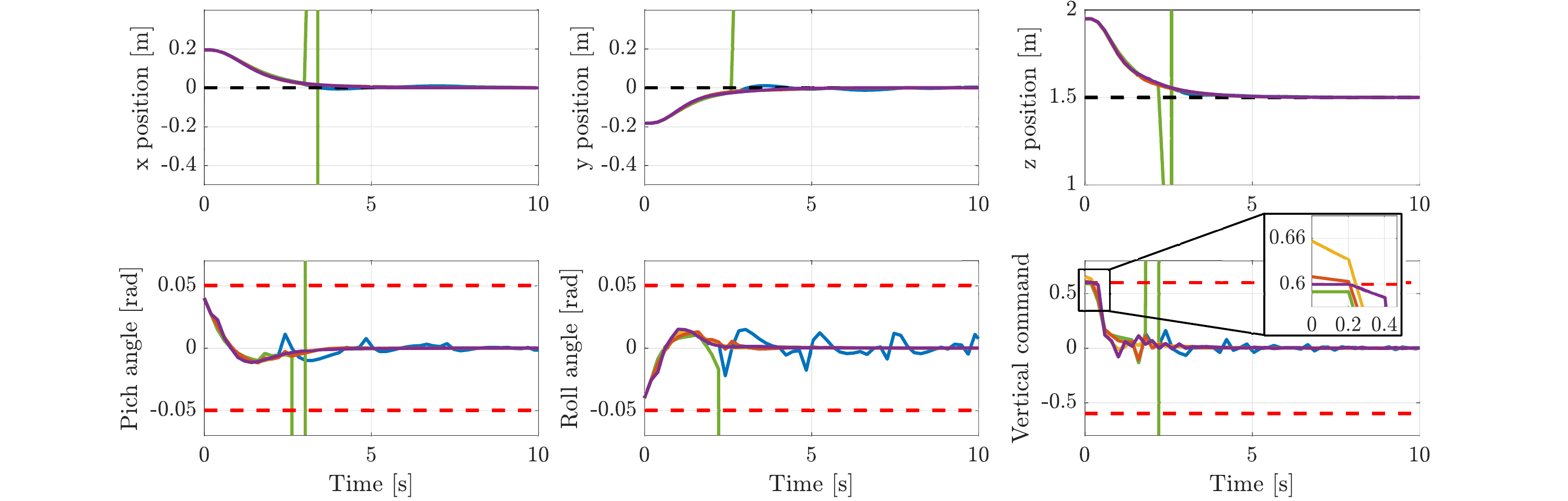}
    \caption{Time-profile of states and control inputs for all five cases.}
    \label{fig:oneExample}
\end{figure*}

\section{Simulation and Experimental Results}\label{sec:Results}

This section aims at evaluating the effectiveness of the proposed discrete-time REAP through extensive simulation studies and real-world experiments, focusing on the position control of a Parrot Bebop 2 drone. Our experimental setup (see Figure \ref{fig:network}) uses the \texttt{OptiTrack} system which consists of ten \texttt{Prime$^\text{x}$ 13} cameras operating at a 120 Hz frequency and providing a 3D precision of $\pm0.02$ millimeters. The computing unit is a $13^{\text{th}}$ Generation $\text{Intel}^{\text{\textregistered}}$ $\text{Core}^{\text{\texttrademark}}$ i9-13900K processor coupled with 64GB of RAM, that runs the \texttt{Motive} software to analyze and interpret camera data, and performs the computations of the proposed discrete-time REAP. The communication between \texttt{Motive} and MATLAB is done through User Datagram Protocol (UDP) using the \texttt{NatNet} service.


\subsection{Simulation Analysis and Comparison Study}\label{sec:Simulation}

\begin{figure}[!t]
    \centering
    \includegraphics[width=1\linewidth]{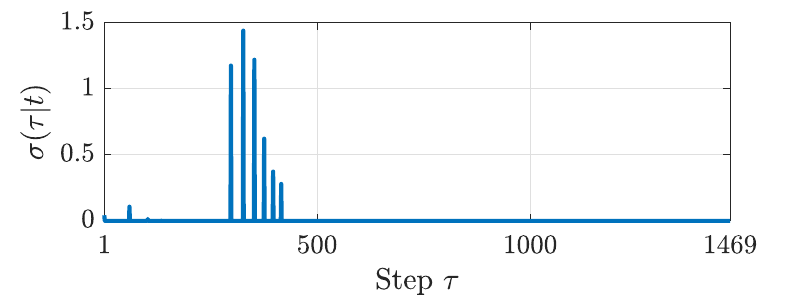}
    \caption{Time profile of the KKT parameter $\sigma(\tau|t)$ with the proposed scheme (i.e., CASE V).}
    \label{fig:sigmaforoneExample}
\end{figure}

Setting the sampling period to 0.2 seconds, the position dynamics of the Parrot Bebop 2 drone can be expressed \cite{amiri2024closed} as a linear system in the form of \eqref{eq:System1}, where $x=[p_x ~ \dot{p}_x ~ p_y ~ \dot{p}_y ~ p_z ~ \dot{p}_z]^\top$ with $p_x,p_y,p_z\in\mathbb{R}$ being  X, Y, and Z positions in the global Cartesian coordinate,  $u=[u_x~u_y~u_z]^\top$  with $u_x,u_y,u_z\in\mathbb{R}$ being control inputs on X, Y, and Z directions, and 
\begin{align*}
&A=\begin{bmatrix}
1 & 0.19895 & 0 & 0 & 0 & 0 \\
0 & 0.98952  & 0 & 0 & 0 & 0 \\
0 & 0 & 1.000 & 0.19963 & 0 & 0 \\
0 & 0 & 0 & 0.99627 & 0 & 0 \\
0 & 0 & 0 & 0 & 1.000 & 0.16816\\
0 & 0 & 0 & 0 & 0 & 0.69946
\end{bmatrix},\\
&B=\begin{bmatrix}
-0.10917348 & 0 & 0  \\
-1.08982035 & 0 & 0  \\
0 & -0.141040918& 0  \\
0 & -1.409531141 & 0  \\
0 & 0 & -0.030967224  \\
0 & 0 & -0.292295416 
\end{bmatrix},\\
&C=\begin{bmatrix}
1 & 0 & 0 & 0 & 0 & 0 \\
0 & 0 & 1 & 0 & 0 & 0 \\
0 & 0 & 0 & 0 & 1 & 0
\end{bmatrix},~~~~~~~~~~~D=\begin{bmatrix}
0 & 0 & 0  \\
0 & 0 & 0 \\
0 & 0 & 0
\end{bmatrix}.
\end{align*}

\begin{figure}[h]
\begin{center}
\includegraphics[width=\linewidth]{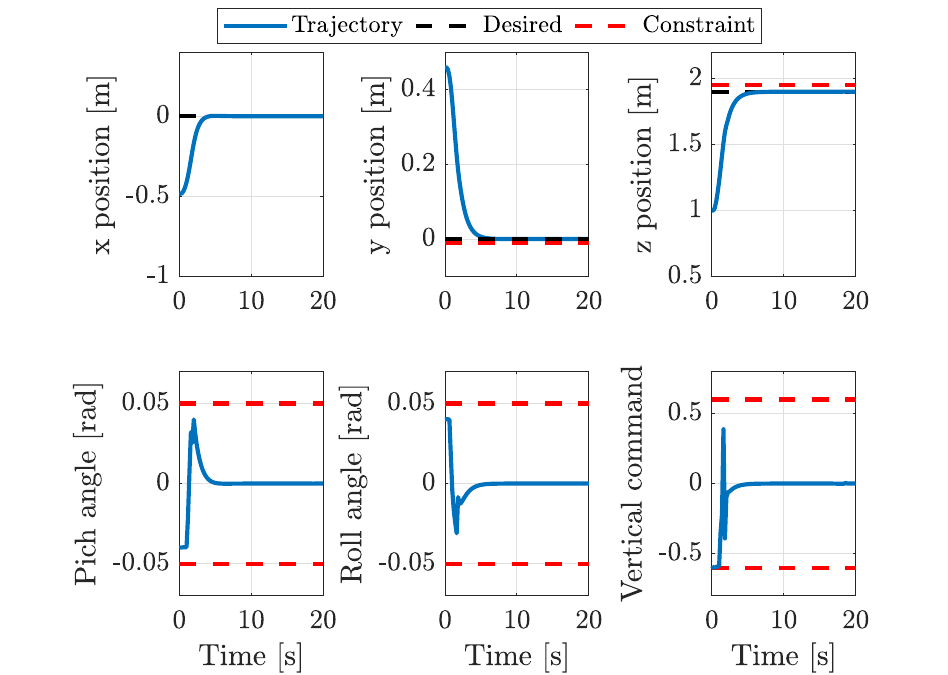}  \\
 \includegraphics[width=1\linewidth]{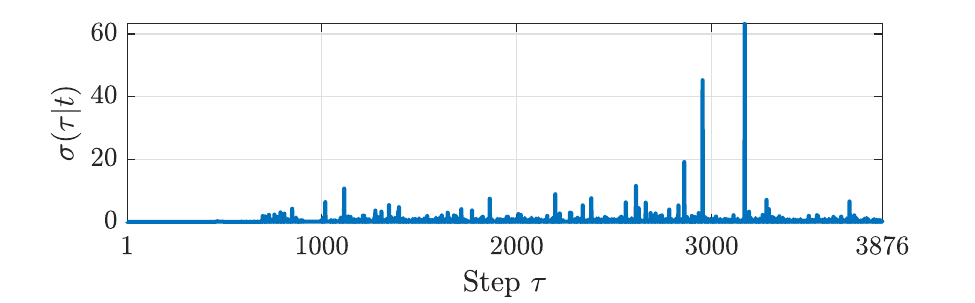}
\caption{Time profile on states and control inputs with the proposed scheme in the presence of constraints on both control inputs and states.  The last row showing the time profile of the KKT parameter $\sigma(\tau|t)$ throughout the entire process.}
\label{fig:WorstCase}
\end{center}
\end{figure}

The control inputs $u_x(t)$, $u_y(t)$, and $u_z(t)$ should satisfy the following constraints for all $t$: $|u_x(t)|\leq0.05$ , $|u_y(t)|\leq0.05$, and $|u_z(t)|\leq0.6$. We set $r = [0~0~0~0~1.5~0]^T$, the weighting matrices $Q_x=\text{diag}\{5,5,5,5,1000,1000\}$ and $Q_u=\text{diag}\{30,20,1\}$, and the length of the prediction horizon $N=10$. Also, the desired reference signal is \( r = \begin{bmatrix} 0 &  0 &  1.5 \end{bmatrix}^T \).

To assess the performance of the proposed discrete-time REAP in guaranteeing constraint satisfaction at all times, we set $d\tau=0.001$ and consider five cases based on the choice of the KKT parameter: i) CASE I, where $\sigma(\tau|t)=2.22\times10^{-16},~\forall t,\tau$ (this case resembles continuous-time implementation of REAP, where $\sigma(\tau|t)\cdot d\tau$ is very small); ii) CASE II, where $\sigma(\tau|t)=0.5,~\forall t,\tau$; iii) CASE III, where $\sigma(\tau|t)=0.05,~\forall t,\tau$; iv) CASE IV, where $\sigma(\tau|t)=0.005,~\forall t,\tau$; and v) CASE V, where $\sigma(\tau|t)$ is set to the upper bound of \eqref{eq:SigmaCondition}. To provide a quantitative comparison, we consider 1,000 experiments with initial condition $x_0=[\alpha_1~0~\alpha_2~0~1.5+\alpha_3~0]^\top$, where in each experiment, $\alpha_1$, $\alpha_2$, and $\alpha_3$ are uniformly selected from the interval [-0.5,0.5]; to ensure a fair comparison, the same initial condition is applied to all five cases.

Table \ref{tab:ExtensiveComparison} reports the percentage of constraint violation for all ceases. As seen in this table, using a fixed KKT parameter causes discretization errors and thus results in constraint violation; the larger the value of the fixed KKP parameter is, the higher the violation percentage will become.

Note that although using a very small KKT parameter (i.e., CASE I, where $\sigma(\tau|t)=2.22\times10^{-16},~\forall t,\tau$) does not lead to constraint violation, it significantly reduces the evolution rate of system \eqref{eq:DisREAP}. Such a reduction in the evolution rate reduces the performance of the MPC scheme to the one of the terminal control law $\kappa(x,r)$ detailed in Subsection \ref{sec:TerminalConstraintSet} which is used to warm start REAP at every time instant (see Remark \ref{Remark:InitialFeasibility}). This point is shown in Figure \ref{fig:oneExample}, where REAP with $\sigma(\tau|t)=2.22\times10^{-16}$ provides a non-smooth control input. The time profile of the KKT parameter $\sigma(\tau|t)$ with the proposed scheme (i.e., CASE V) is illustrated in Figure \ref{fig:sigmaforoneExample}.


\begin{figure}[!t]
\begin{center}
\includegraphics[width=.8\linewidth]{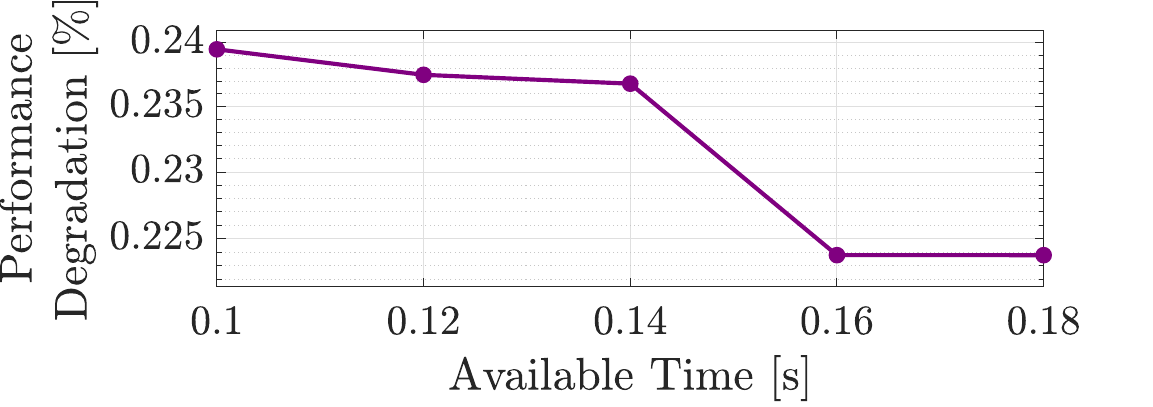}  \\
\caption{Impact of the available time for REAP's execution on the performance degradation.}
\label{fig:Suboptimality}
\end{center}
\end{figure}

\begin{table}[!t]
\caption{Comparison study for the violation}
\centering
\begin{tabular}{c|c}
       \hline
         KKT Parameter &  Violation Percentage [$\%$]\\
       \hline\hline
      CASE I: $\sigma(\tau|t)=2.22\times10^{-16}$  & 0\\
       \hline
      CASE II: $\sigma(\tau|t)=0.5$   & 100\\
        \hline
       CASE III:h $\sigma(\tau|t)=0.05$  &85.2\\
        \hline
       CASE IV: $\sigma(\tau|t)=0.005$   &48.9\\
        \hline
       CASE V: $\sigma(\tau|t)$ as in \eqref{eq:SigmaCondition}  &0\\
        \hline
    \end{tabular}
    \label{tab:ExtensiveComparison}
\end{table}

\begin{figure*}[!t]
\begin{center}
\includegraphics[width=.8\columnwidth]{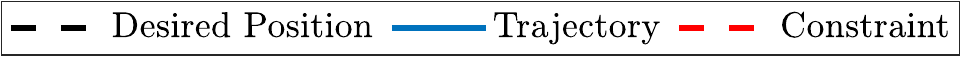}\\
\includegraphics[width=1\linewidth]{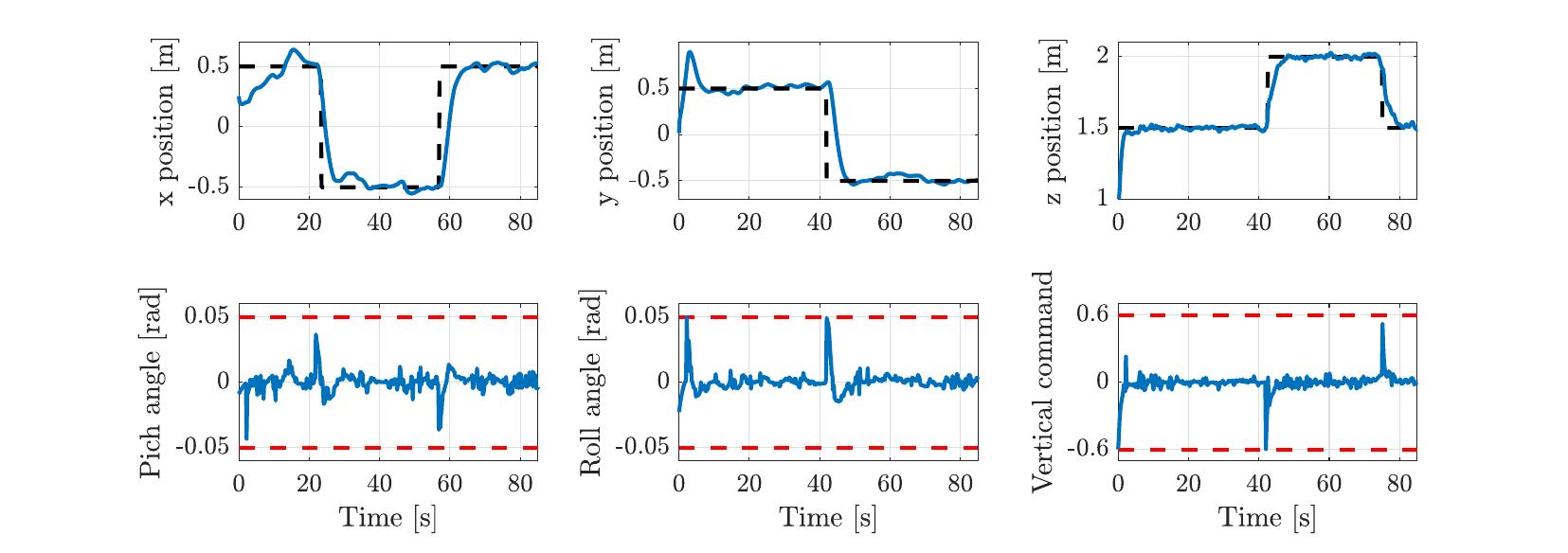} 
\caption{Experimental results\textemdash time-profile of states and control inputs of the Parrot Bebop 2 drone with discrete-time REAP given in \eqref{eq:DisREAP}.}
\label{fig:EXp}
\end{center}
\end{figure*}

\begin{figure}[!t]
    \centering
    \includegraphics[width=1\linewidth]{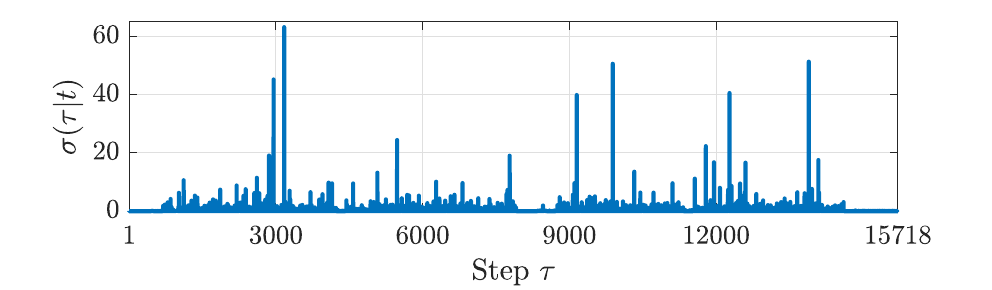}\\
    \includegraphics[width=.45\linewidth]{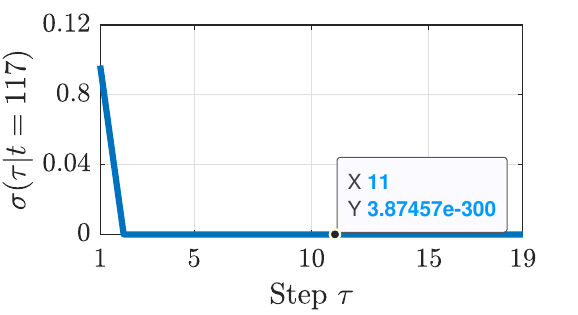}\includegraphics[width=.45\linewidth]{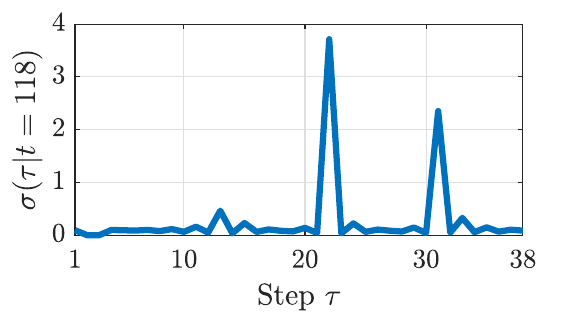}
    \caption{Evolution of the KKT parameter $\sigma(\tau|t)$; top figure shows the evolution across all time instants, while the bottom figures show the evolution within two typical time instants.}
    \label{fig:sigma}
\end{figure}

\subsection{Constraints on Control Inputs and States} 
In this subsection, we evaluate the performance of the proposed scheme in the presence of constraints on both the control inputs and the states. We assume that the control inputs are constrained as \( |u_x(t)| \leq 0.05 \), \( |u_y(t)| \leq 0.05 \), and \( |u_z(t)| \leq 0.6 \) for all \( t \), while the constraints on states are \( y(t) \geq -0.01 \) and \( z(t) \leq 1.95 \). The desired reference signal is \( r = \begin{bmatrix} 0 &  0 &  1.9 \end{bmatrix}^T \). 

Figure \ref{fig:WorstCase} presents the simulation results, with the last row showing the time profile of the KKT parameter $\sigma(\tau|t)$ throughout the entire process. As seen in the figure, the proposed scheme effectively steers the drone to the desired position while satisfying constraints on both control inputs and states.

\subsection{Suboptimality Assessment}
To assess the suboptimality of the proposed scheme, we compare the obtained solution with the optimal solution which can be achieved by running MPC on a more powerful processor. Figure \ref{fig:Suboptimality} illustrates the performance degradation relative to the available time, which is the duration that REAP is executed at each time instant. Note that the performance is calculated as $\text{Performance}=\sum_{t}\left\Vert x(t)-\bar{x}_r\right\Vert_{Q_x}^2+\left\Vert u(t)-\bar{u}_r\right\Vert_{Q_u}^2$. 

As shown in Figure \ref{fig:Suboptimality}, the performance degradation remains below 0.3\%, even when the available time for REAP's execution is half of the sampling period. This demonstrates the effectiveness of the proposed scheme in maintaining acceptable performance despite limitations on available time.

\subsection{Experimental Analysis} 
This subsection aims at experimentally validating the proposed discrete-time REAP by investigating its performance in controlling the position of the Parrot Bebop 2 drone.

Experimental results are presented in Fig. \ref{fig:EXp}. As seen in this figure, the developed discrete-time REAP effectively steers the Parrot Bebop 2 drone to the desired position, while satisfying the constraints at all times. The discrete-time REAP shows sufficient robustness to early termination by executing different numbers of computation steps (with the mean value of 37) depending on the available time at each time instant, with minimum performance degradation as  shown in Figure \ref{fig:EXp}.

Figure \ref{fig:sigma} presents the time-profile of the computed KKT parameter, where the second row highlights the point discussed in remark \ref{remark:nonzero}. Figure \ref{fig:sigma} reveals that, whenever the constraints are far from being violated, the developed discrete-time REAP increases the KKT parameter to speed up the evolution of system \eqref{eq:DisREAP} and consequently improve its convergence. Note that $\sigma(\tau|t)\in[0,63.2137]$ during the experiment, with a mean value of $0.2392$ across all time instants and computation steps.

\section{Conclusion}\label{sec:Conclusion}
REAP is a systematic approach designed to address limited computing capacity for MPC implementations. While prior work provides theoretical guarantees that
are developed in continuous time, REAP must be
implemented in discrete time in real-world applications. This paper investigated the theoretical guarantees associated with discrete-time implementation of REAP. It was analytically shown that by computing updates of the control sequence in discrete time and adaptively adjusting the KKT parameter, one can maintain the REAP's anytime feasibility and convergence properties when its computations are performed in discrete time. The effectiveness of the proposed methodology was evaluated through extensive simulation and experimental studies. This paper also developed a MATLAB package that allows researchers to easily utilize the proposed methodology in their research.

\appendix
\renewcommand\thefigure{A.\arabic{figure}}
\setcounter{figure}{0}


\begin{figure}[b]
    \centering
    \includegraphics[width=0.5\linewidth]{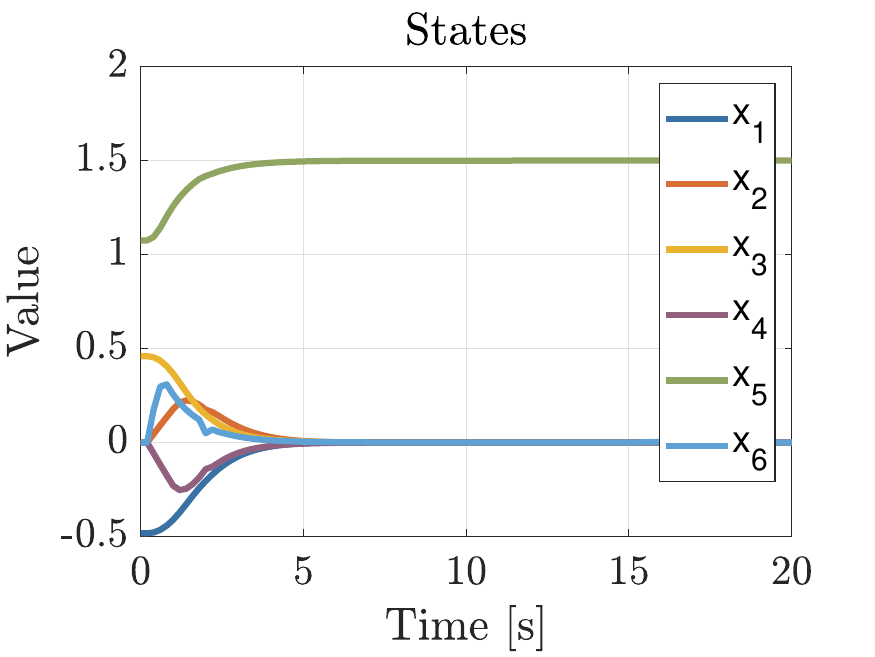}\includegraphics[width=0.5\linewidth]{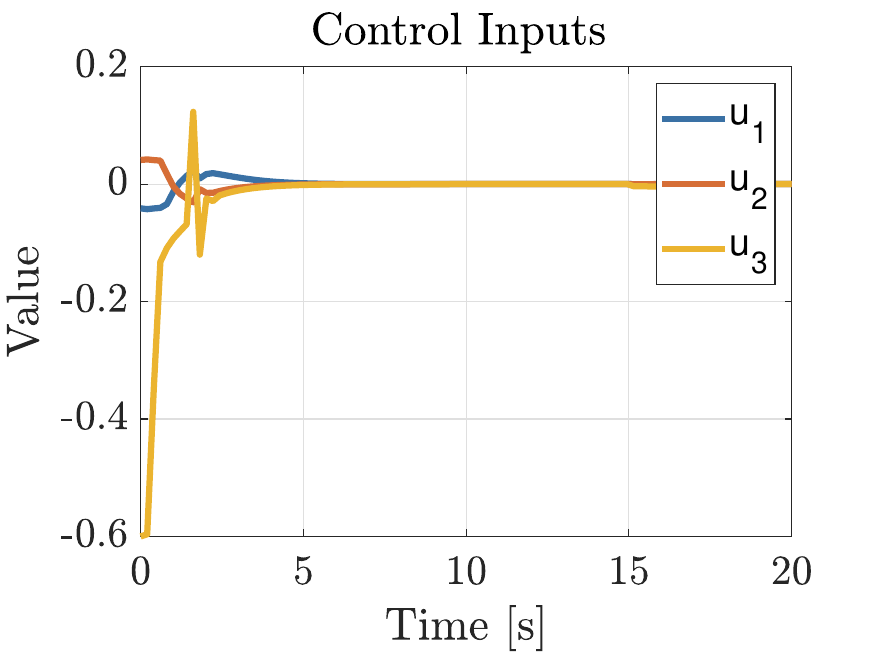}
    \caption{Time-profile of system states and applied control input generated by the ``discreteREAP" package.}
    \label{fig:Output}
\end{figure}

\section{``DiscreteREAP" MATLAB Package}
We have developed the ``DiscreteREAP" MATLAB package to allow researchers to adopt the proposed methodology in their research. This package is available at the following URL: \href{https://github.com/mhsnar/DiscreteREAP.git}{https://github.com/mhsnar/DiscreteREAP.git}.

\medskip
\noindent\textbf{Features:} Key features of the ``DiscreteREAP" package are:
\begin{itemize}
    \item Efficient Control Design: The package streamlines the process of designing control schemes by providing built-in functions for system modeling, controller synthesis, and simulation.

 \item Constraint Handling: Users can easily incorporate constraints in the form of \eqref{eq:allConstraints} into their control designs to address system performance and safety, even in the presence of limitations on computational resources.

 \item Simulation Capabilities: The package enables users to simulate the behavior of their control systems and assess their performance.

 \item Tutorial Support: A detailed tutorial guides users through various functionalities of the package, helping them understand and leverage its capabilities effectively.
\end{itemize}

\medskip
\noindent\textbf{Examples:} 
To facilitate the usability of the package, two examples are provided under the main directory. The first example focuses on controlling the position of the Parrot Bebop 2 drone, which is detailed in Section \ref{sec:Results}. The second example considers a system in the form of \eqref{eq:System1} with the following matrices (this system has been discussed in \cite{Limon2008,SSMPC}):
\begin{align*}\label{eq:NumericalExample}
A= \begin{bmatrix} 1 & 1 \\ 0 & 1 \end{bmatrix},~~~~B = \begin{bmatrix} 0.0 & 0.5 \\ 1.0 & 0.5 \end{bmatrix},~~~~C = \begin{bmatrix} 1 & 0 \end{bmatrix},~~~~D = 0,
\end{align*}
where states and control inputs are constrained as $\left\vert x_i(t)\right\vert\leq5$ and $\left\vert u_i(t)\right\vert\leq10,~i\in\{1,2\}$ for all $t$. In this example, we set $r=4.85$, and $Q_x=\text{diag}\{1,1\}$ and $Q_u=\text{diag}\{1,1\}$. Time-profiles of the states and control inputs are shown in Figure \ref{fig:Output}.





\bibliographystyle{elsarticle-num} 

\bibliography{ref}  


\end{document}